\documentclass[a4paper,10pt]{article}
\usepackage{amscd}
\usepackage{amsmath}
\usepackage{bbding}
\usepackage{amsfonts}
\usepackage{cite}
\usepackage{mathrsfs}
\usepackage{stmaryrd}
\usepackage{graphicx,latexsym,amsmath,amssymb,amsthm,enumerate}
\usepackage{bm}
\usepackage{color}
\newcommand{\comp}{\ensuremath\mathop{\scalebox{.6}{$\circ$}}}

\theoremstyle{definition}
\newtheorem{theorem}{Theorem}[section]
\newtheorem{lemma}{Lemma}[section]
\newtheorem{corollary}{Corollary}[section]

\newtheorem{definition}{Definition}[section]

\newtheorem{remark}{Remark}[section]

\def\acknowledgement{\par\addvspace{17pt}\small\rmfamily
	\trivlist\if!\ackname!\item[]\else
	\item[\hskip\labelsep
	{\bfseries\ackname}]\fi}

\def\ackname{Acknowledgements}
\hfuzz=\maxdimen
\begin{document}
\title{{\LARGE\bf{On the supporting quasi-hyperplane and separation theorem of geodesic convex sets with applications on Riemannian manifolds}}}
\author{Li-wen Zhou$^{a,b}$, Ling-ling Liu$^{a,b}$, Chao Min$^{a,b}$, Yao-jia Zhang$^{a,b}$
	 and Nan-Jing Huang$^{c}$\footnote{E-mail address: njhuang@scu.edu.cn; nanjinghuang@hotmail.com}\\
	$^a${\small\it School of Sciences, Southwest Petroleum University, Chengdu, 610500, P.R. China}\\
	$^b${\small\it Institute for Artificial Intelligence, Southwest Petroleum University, Chengdu 610500, Sichuan, P. R. China}\\
$^c${\small\it Department of Mathematics, Sichuan University, Chengdu,
	Sichuan 610064, P.R. China}}
\date{ }
\maketitle
\vspace*{-9mm}
\begin{center}
\begin{minipage}{5.6in}
\noindent{\bf Abstract.} In this paper, we first establish the separation theorem between a point and a locally geodesic convex set and then prove the existence of a supporting quasi-hyperplane at any point on the boundary of the closed locally geodesic convex set on a Riemannian manifold. As applications, some optimality conditions are obtained for optimization problems with constraints on Riemannian manifolds.
\end{minipage}
\end{center}

\section{Introduction} \noindent

In recent years, some large-scale, high-dimensional, nonconvex and nonsmooth optimization problems from artificial intelligence machine learning need to be solved. However, traditional iterative algorithms used to solve these problems is not satisfactory. In order to solve these optimization problems in engineering applications and theoretical research, researchers try to embed the original complex problems from Euclidean space to the differential manifold by selecting appropriate metrics, linearizing the nonlinear problem, transforming the nonconvex programming into geodesic convex programming, and reducing the dimension of high-dimensional big data through isometric mapping, so as to simplify the problem. In general, the main difference between the manifold optimization problem and the classical optimization problem is that the differential manifold is not a linear space. The natural idea is to extend the related concepts and methods of linear space to Riemannian manifolds by using geodesics (see \cite{WK,tr,CU,do} for more details). Actually, the concepts and techniques in the last decades, which fit in Euclidean spaces, have extended to the nonlinear framework of Riemannian manifolds (see, for example, \cite{jka, ugmt,tr,OP2,OP3,OP1,llly,llm,llm2,nemeh,colao,tang,zhou1,zhou2,li1,li2,ch}, and the references therein).

Generally speaking, convexity plays a very important role in optimization theory and related problems. On this account, introducing geodesic convexity into manifolds to study optimization problems is an interesting research direction. The manifold optimization problems is not a simple generalization of the classical optimization problem in Euclidean space. Although the differential manifold is a local differential homeomorphism in an Euclidean space, the key to this problem is that the convexity is not topological homeomorphism invariance, and it is much more complex on the general manifold than Euclidean space. Therefore, the results involving convexity are essentially different from the classical ones in Euclidean space. For example, on Hadamard manifolds, N\'{e}meth \cite{nemeh} proved the Brouwer fixed point theorem under geodesic convexity, and obtained the existence of solutions for a class of variational inequality problems; Colao  et al. \cite{colao} and Zhou et al. \cite{zhou1} proved the existence of the solution for solving one equilibrium problem. We also characterized the geodesic pseudo-convex combination and proved that the KKM mapping on the differential manifold has the property of non empty intersection \cite{zhou2,zhou3}. We believe that the further study of geodesic convexity and related problems on differential manifolds is an innovative work different from the classical conclusion in linear space.

Besides, many scholars have paid attention to the optimality conditions and related problems of nonlinear programming problems on manifolds (for example, see \cite{bh,ror,ch1,ugmt} and its references). On the basis of these above research, it is interesting and necessary to further study the constrained manifold optimization problem. This is because the Karush-Kuhn-Tucker (KKT) condition of constrained optimization in Euclidean linear space is derived from the separation theorem of convex sets. The separation of convex sets, the existence and properties of supporting hyperplanes play important roles in constrained optimization problems. In classical theory, the separation theorem is derived from the algebraic description of the projection of a point on a closed convex set. The algebraic expression is a form of variational inequality, which means that variational inequality is also sufficient in this problem. In addition, the results in related research also include the characterization of various constraint features and Fakas' Lemma. Moreover, the above studies can be characterized equivalently by the theory and geometric properties of cones, which means that the polar cones, tangent cones and normal cones defined on tangent spaces are effective tools for studying the constraints and optimality conditions in manifold optimization problems. The main idea of this paper is to extend the above conclusions to nonconvex case, and discuss the concept and properties of cones on tangent spaces by using the geometric properties of manifolds. Specifically, we study the separation theorem between a point and a locally geodesic convex set on a Riemannian manifold, and prove that there exists a supporting quasi hyperplane at any point on the boundary of a closed locally geodesic convex set. As an application, we obtain a class of optimality conditions for constrained optimization problems on Riemannian manifolds.

The rest of this paper is organized as follows. Section 2 presents some necessary concepts, lemmas and propositions. In Section 3, a quasi-hyperplane is defined and a result is given to show that a point can be separated from a local geodesic convex set by this quasi-hyperplane. A proof is also obtained to indicate that there is a supporting quasi-hyperplane at any point on the boundary of a local geodesic convex set. Before summarizing this paper, in Section 4, some concepts and properties of tangent cones, polar cones and normal cones on manifolds are discussed. As applications, the optimality conditions for a class of constrained manifold optimization problems have been established.

\section{Preliminaries}\noindent
\setcounter{equation}{0}
In this section, we recall some definitions and basic properties about the geometry of the manifolds used in this paper, which can be found in many introductory books on Riemannian and differential geometry (see, for example, \cite{do,gro}).

Let $M$ be a smooth, simply connected, Hausdorff $m$-dimensional manifold, which has a countable atlas. By given $x \in M$, the tangent space of $M$ at $x$ denoted by $T_x M$ is a linear subspace, and the tangent bundle of $M$ defined by
$$TM:=\bigcup_{x\in M} T_x M,$$
is a naturally manifold. A vector field $V$ on $M$ is a mapping of $M$ into $TM$, which is associated with each point $x\in M$ and a vector $V(x) \subset T_x M$. The manifold $M$ can be endowed with a Riemannian metric to become a Riemannian manifold. Then the scalar product and the norm denoted by $\langle \cdot , \cdot \rangle$ and  $\|\cdot\|$ respectively are defined on $T_x M.$ Recall that the length of a piecewise smooth curve $\gamma : [a, b]\rightarrow M$ joining $x=\gamma(a)$ to $y=\gamma(b)$ is defined as
$$L(\gamma):=\int_a^b \|\dot{\gamma}(t)\|dt.$$
For any $x, y\in M$, the Riemannian distance $d(x, y)$ is the minimal length of all the curves joining $x$ to $y$ (see \cite{CU}, p.22).

A vector field $V$ is said to be parallel along a smooth curve $\gamma\in M$ if and only if $\nabla_{\dot{\gamma}(t)} V=0,$ where $\nabla$ is the Levi-Civita connection associated with $(M, \langle\cdot ,\cdot \rangle ).$ The smooth curve $\gamma$ is said to be a geodesic if and only if $\nabla_{\dot{\gamma}(t)}{\dot{\gamma}(t)}=0.$ The gradient of a differentiable function $f:M\to \mathbb{R}$, $\mbox{grad}f$, is a vector field on $M$ defined through $df(X) = \langle \mbox{grad}f, X\rangle=X(f)$, where $X$ is also a vector field on $M$. A geodesic joining $x$ to $y$ in $M$ is said to be minimal if and only if its length is equal to $d(x, y)$.

As Cheeger and Cromoll have shown in \cite{cg}, any closed locally convex and connected set $C\subset M$ bears the structure of an imbeded $k$-dimensional submanifold with smooth totally geodesic interior of $C$.

\begin{definition}
The interior of $C$, denoted by $\mbox{int}(C)$ is the union of all submanifolds of $M$ contained in $C,$ and the boundary of $C$ is $\mbox{bd}(C):=C\backslash \mbox{int}(C).$
\end{definition}

\begin{definition}(\cite{gro}, \cite{CU}, p.17)
The exponential mapping $exp_p: T_p M\rightarrow{M}$ is defined by $exp_p \nu:=\gamma_{\nu}(1)$, where $\gamma_{\nu}$ is the geodesic
defined by its position $p$ and velocity $\nu$ at $p$.
\end{definition}

%We use $\nabla$ to introduce an isometry $P_{\gamma,\cdot,\cdot}$ on the tangent bundle $TM$ along $\gamma$, which is called the parallel transport and defined as
%$$P_{\gamma,\gamma(b),\gamma(a)}(v):=V(\gamma(b)),\quad \forall a, b\in R
%\mbox{ and }v\in T_{\gamma(a)}M,$$
 %where $V$ is the unique vector field satisfying $\nabla_{\gamma'(t)}V=0$ for all $t$ and $V(\gamma(a))=v$. Without any confusion, $\gamma$ can be omitted when it is a minimal geodesic joining $x$ to $y.$

%\begin{definition}
%A Hadamard manifold $M$ is a simply-connected complete Riemannian
%manifold of non-positive sectional curvature.
%\end{definition}

%\begin{lemma}{\rm (\cite{do}, p.149, Theorem 3.1, Cartan-Hadamard theorem)}\label{l0:this}
%Let $M$ be a Hadamard manifold, then the universal cover of $M$ is a convex
%geodesic space with respect to the induced length metric $d$. In
%particular, any two points of the universal cover are joined by a
%unique geodesic.
%\end{lemma}

%\begin{proposition}{\rm (\cite{do}, p.149, Theorem 3.1)}\label{p2.1}
%Let $M$ be a Hadamard manifold and $p\in M$. Then, $exp_p: T_p M\rightarrow M$ is a diffeomorphism, and for
%any two points $p, q\in M$, there exists a unique minimal geodesic $\gamma_{p, q}(t)=exp_p t exp_p^{-1}q$ for all $t\in [0, 1]$ joining
%$p$ to $q$.
%\end{proposition}

It follows from (\cite{gro}, p.63-64) the definition of geodesic flow that there is an open starlike neighborhood $\mathscr{U}$ of zero vectors in $TM$ and an open neighborhood $\mathscr{V}$ in $M\times M$ such that the exponential map $u\mapsto \exp_{\pi(u)}u$ is smooth on $\mathscr{U}$, and
\begin{equation}\label{geoflow}
	\pi\times \exp: \mathscr{U}\rightarrow \mathscr{V}
\end{equation}
is a diffeomorphism, where $\pi$ is the projection of $TM$ onto $M$. Then, for $u\in \mathscr{U}$, the map
$t\rightarrow \exp tu, 0\leq t\leq 1$, describes a minimal geodesic from $\pi(u)$ to $\exp u$, and this
minimal geodesic is the only one between these points. The diffeomorphism inverse to (\ref{geoflow}) will be denoted by $\Phi: \mathscr{V}\rightarrow \mathscr{U},$ so for any given $(p,q)\in \mathscr{V}, \Phi(p, q)\in \mathscr{U}$ is the initial vector of the unique minimal geodesic from $p$ to $q$ of length $d(p, q)$, which implies that the smoothness of $(p, q)\mapsto|\Phi(p, q)|^2$ on $\mathscr{V}$. Then we can define the geodesic flow as follows:

\begin{definition}
The geodesic flow $\Omega$
\begin{equation}\label{w1}
\Omega: \mathscr{W}\rightarrow TM
\end{equation}
is defined by	
\begin{equation}\label{w2}
	\Omega(u,t):=\partial(\exp tu)/\partial t,
\end{equation}
where $\mathscr{W}$ consisting of all $(u, t)\in TM\times R$ with $tu\in \mathscr{U}$. The open set $\mathscr{W}$ is a neighborhood of $TM\times{0}$, and $\Omega$ is smooth on $\mathscr{W}$.
\end{definition}

%\begin{proposition}{\rm(\cite{do}, p.150, Lemma 3.2)}\label{p2.2}
%The exponential mapping and its inverse are continuous on a Hadamard manifold.
%\end{proposition}
\begin{lemma}{\rm(\cite{llm2}, Lemma 2.4)}\label{l2.1}
Let $x_0\in M$ and $\{x_n\}\subset M$ such that $x_n\rightarrow x_0$. Then, the following assertions hold.
\begin{itemize}
\item[(i)] For any given $y\in M$,
$exp^{-1}_{x_n} y\rightarrow exp^{-1}_{x_0} y$ and $exp^{-1}_y x_n \rightarrow exp^{-1}_y x_0;$

\item[(ii)] If $\{v_n\}$ is a sequence such that $v_n \in T_{x_n} M$ and $v_n\rightarrow v_0$, then $v_0\in T_{x_0} M$;

\item[(iii)] For given sequences $\{u_n\}$ and $\{v_n\}$ satisfying $u_n, v_n\in T_{x_n} M$, if $u_n\rightarrow u_0$ and $v_n \rightarrow v_0$ with $u_0, v_0\in T_{x_0} M$, then
$$\langle u_n, v_n\rangle\rightarrow \langle u_0, v_0\rangle.$$
\end{itemize}
\end{lemma}

\begin{definition}(\cite{Rolf})
A subset $C\subset M$ is said to be weakly geodesic convex if and only if for any two points $x, y\in C$, there is a minimal geodesic $\gamma: [a, b]\rightarrow M$ from $x$ to $y$ lying in $C$, that is, if $\gamma : [a, b]\rightarrow M$ is a geodesic such that $x=\gamma (a)$ and $y=\gamma (b)$, then $\gamma(t)\in C$ for all $t\in [a, b].$ If such a minimal geodesic is unique within $C$ then $C$ is said to be geodesic convex.
\end{definition}

\begin{definition}(\cite{Rolf, rinow})
A subset $C\subset M$ is said to be strongly geodesic convex if and only if for any two points $x, y\in C$, there is just one minimal geodesic from $x$ to $y$ and it is in $C.$ In this case the image of the geodesic is denoted by $\gamma_{x, y}$ and is called geodesic segment.
\end{definition}

\begin{definition}
For each $p\in M$, there is a largest $r(p)>0,$ which is called geodesic convex radius, such that for $0<r\leq r(p)$, all the open ball $\mbox{Bal}_M(p,r)$ in $M$ is strongly geodesic convex and each geodesic in $\mbox{Bal}_M(p,r)$ is minimal.
\end{definition}

\begin{remark}(\cite{Rolf})
There is a largest $d(p)>0$ such that, for $0<d\leq d(p),$ the Euclidean ball $\mbox{Bal}_{T_p M}(0,d)$ in $T_p M$ is diffeomopphic to $\mbox{Bal}_M(p,r)$ via the exponential map $\exp_p.$
\end{remark}

\begin{definition}\label{rodf1}
A subset $C\subset M$ is said to be locally geodesic convex if and only if for any $p$ in the closure of $C$, namely, $\mbox{cl}(C)$, there is a positive $\varepsilon(p)<r(p)$ such that $C\cap \mbox{Bal}_{M}(p,\varepsilon(p))$ is strongly geodesic convex.
\end{definition}

From now on, let a Riemannian manifold $M$ be endowed by a
Riemannian metric $\langle\cdot, \cdot\rangle$ with corresponding
norm denoted by $\|\cdot\|,$ and a subset $C\subset M$ is locally geodesic convex. Now we recall some useful results about the geodesic on a Riemannian manifold.

\begin{lemma}(\cite{cg}, Lemma 2.3)
	For each compact subset $A\subset M$ there is an $\varepsilon_A>0$ such that, for all $p\in A$ and $0<r\leq \varepsilon_A\leq r(p),$	if $\gamma:[0, b]\rightarrow \mbox{Bal}_{M}(p,r)$ is a nonconstant geodesic, and $\gamma_0:[0, 1]\rightarrow \mbox{Bal}_{M}(p,r)$ is the minimal geodesic from $p$ to $\gamma(0)$, with $\left\langle \dot{\gamma}(0),\dot{\gamma}_0(1) \right\rangle_{T_{\gamma(0)}M} \geq 0,$ then $t\mapsto d(\gamma(t), p)$ is strictly increasing on $[0, b].$
\end{lemma}

\begin{lemma}(\cite{Rolf}, Theorem 1)\label{roth1}
	For each locally geodesic convex set $C\subset M,$ if $C$ is closed, then there is an open neighborhood $U$ of $C$ such that, for each $q\in U$, there exists a unique $q^*\in C$ such that
	$$d(q, q^*)=\inf\{d(q, p)|p\in C\}=:d(q, C),$$
	which is a unique minimal geodesic from $q$ to $q^*$ lying in $U.$
\end{lemma}

\begin{remark}
	The map $q\mapsto q^*$ is said to be the metric projection onto $C$ and the projection $q^*$ is denoted as $\mbox{Pr}_{C}(q)$.  The previous lemma gives the existence of projection of any point in the geodesic convex open neighborhood of a closed locally geodesic convex set $C.$
\end{remark}

The following lemma shows that the metric projection can be characterized by variational inequality in a tangent space.

\begin{lemma}(\cite{Rolf}, Lemma 2)\label{roth2}
	Let closed set $C$ and $U$ be defined in Lemma \ref{roth1} and $\varepsilon(\cdot)$ be defined in Definition \ref{rodf1}. For each $x\in C$, let $V(x)\subset T_x M$ be a vector field defined by
$$V(x)=\left\{v_x\in T_x M|\exp_{x}t\frac{v_x}{\|v_x\|}\in \mbox{int}(C)\;\mbox{for some positive}\; t<\varepsilon(p)\right\}.$$
Then the following statements hold:
\begin{itemize}
	\item [(i)] For any given $q\in U\backslash C$, let $q^*=\mbox{Pr}_{C}(q)$ be the projection of $q$. Then for each $v_{q^*}\in V(q^*),$ one has
	\begin{equation}\label{eq01}
		\left\langle u,v_{q^*} \right\rangle_{T_{q^*}M} \leq 0,
	\end{equation}
where $u=\dot{\gamma}_{q^*,q}(0)=\exp^{-1}_{q^*}q\in T_{q^*}M$ is the initial vector of $\gamma_{q^*,q}$.

\item [(ii)] Let $p\in C$ and $0\neq u\in T_p M$ satisfy $\left\langle u,v_p \right\rangle_{T_{p}M} \leq 0$ for all $v_p\in V(p)$.  If $t_1>0$ such that $\exp_p tu\in U$ on $[0,t_1]$, $\gamma(t)=\exp_p tu$ is minimal geodesic on  $[0,t_1]$, and $\mbox{Pr}_{C}(\exp_p tu)=p$ on  $[0,t_1]$,  then $\exp_p tu\in U\backslash C$ on $(0, t_1].$
\end{itemize}	
\end{lemma}

 The conclusions presented in Lemma \ref{roth2} can also be described by the cone on the tangent space. To this end, we first recall the definition of tangent cone of $C$ as follows.

\begin{definition}
	Let $C$ be a subset of $M$.  For any $p\in C\subset M,$ the tangent cone is defined as
	$$T_C(p):=\left\{v\in T_p M\left|\exp_p t\frac{v}{\|v\|}\in \mbox{int}(C)\,\,\mbox{for some positive}\; t<\varepsilon(p)\right.\right\}\cup\{0\}=V(p)\cup \{0\}.$$
	The polar cone of $T_C(p)$ is said to be a normal cone, which is given by
	$$N_C(p)=T_C(p)^*=\left\{u\in T_p M|\left\langle u, v \right\rangle_{T_p M}\leq 0, \forall v\in T_C(p)\right\}.$$
\end{definition}

\begin{remark}
	A polar cone is always a closed convex set in a tangent space. One has the following equivalences:
	$$p\in \mbox{int}(C)\Leftrightarrow T_C(p)=T_p M\Leftrightarrow N_C(p)=\{0\}.$$
\end{remark}

Thus,  Lemma \ref{roth2} can be rewritten as follows.

\begin{lemma}\label{roth3}
	Let $C, U$ be the same as in Lemma \ref{roth1}, $\varepsilon(\cdot)$ be the same as in Definition \ref{rodf1}, and for any $q\in U\backslash C$, $q^*=\mbox{Pr}_{C}(q)$ be the projection of $q$. If $u$ is the initial vector of $\gamma_{q^*, q},$ then $u\in N_C(q^*).$
\end{lemma}

We also need the following definition and proposition.
\begin{definition}\label{d2.9}(\cite{CU}, p. 61)
	A real-valued function $f: M\rightarrow \mathbb{R}$ defined on $C$ is said to be geodesic convex if and only if for any geodesic $\gamma$ of $C$, the composition function
	$f\comp\gamma : \mathbb{R}\rightarrow \mathbb{R}$ is convex, i.e.,
	$$(f\comp\gamma)(ta+(1-t)b)\leq t(f\comp\gamma)(a)+ (1-t)(f\comp\gamma)(b)$$
	for any $a, b\in \mathbb{R}$ and $0\leq t\leq 1$.
\end{definition}

%\begin{remark}
%	By Proposition \ref{p2.1}, on a Hadamard manifold $M$, a mapping \\$f: M\rightarrow \mathbb{R}$ is geodesic convex iff it satisfies
%	$$f(exp_x t exp_x^{-1}y)\leq (1-t)f(x)+tf(y)$$
%	for all $x, y\in M$ and $t\in [0, 1].$
%\end{remark}

\begin{lemma}\label{p2.3}{\rm(\cite{41} p. 222)}
	Let $d : M \times M \rightarrow \mathbb{R}$ be the geodesic distance function. Then $d$ is a geodesic convex function with respect to the product Riemannian metric, that is, for any given pair of geodesics $\gamma_1 : [0, 1]\rightarrow M$ and $\gamma_2 : [0, 1]\rightarrow M$, the following inequality holds for all $t \in [0, 1]:$
	$$d\left(\gamma_1(t), \gamma_2(t)\right)\leq(1-t)d\left(\gamma_1(0), \gamma_2(0)\right)+t d\left(\gamma_1(1), \gamma_2(1)\right).$$
	In particular, for each $y \in M$, the function $d(\cdot , y) : M\rightarrow \mathbb{R}$ is a geodesic convex function.
\end{lemma}

%\begin{defn}
%The subdifferential of a function $f: M\rightarrow R$ is the set-valued mapping $\partial f: M\rightarrow 2^{TM}$ defined by
%$$\partial f(x)=\{ u\in T_x M: \langle u, exp_x^{-1} y\rangle\leq f(y)-f(x), \forall y\in M\}, \quad \forall x\in M,$$
%and its elements are called subgradients. The subdifferential $\partial f(x)$ at a point $x\in M$ is a closed geodesic convex (possibly empty) set.
%Let $\mathscr{D}(\partial f )$ denote the domain of $\partial f$ defined by
%$$\mathscr{D}(\partial f ) =\{x\in M: \partial f(x)\neq \emptyset\}.$$
%\end{defn}

%The existence of subgradients for convex functions is guaranteed by the following proposition.

%\begin{proposition}\label{p2.2}(see \cite{OP3})
%Let M be a Hadamard manifold and $f : M\rightarrow R$ be geodesic convex. Then, for any $x\in  M$, the subdifferential $\partial f $
%of $f$ at $x$ is nonempty, that is, $\mathscr{D}(\partial f )=M$.
%\end{proposition}

\section{Separation theorem on $M$}

It is well known that the separation theorem of convex sets can be obtained by using the projection property in $\mathbb{R}^n$. In this section, we first define a quasi-hyperplane on a Riemannian manifold and then prove a similar separation theorem of geodesic convex sets by Lemma \ref{roth2}.
\begin{definition}
Let $U$ be the same as in Lemma \ref{roth1} such that, for any given points $p, q\in U$,  $u=\exp^{-1}_p q$ is the initial vector of the minimal geodesic $\gamma_{p,q}$. Then, for any given real number $\alpha$, the set
$$H(p,q,\alpha)=\left\{a\in U| \left\langle u, \exp^{-1}_p a\right\rangle_{T_p M}=\alpha\right\}$$
is said to the quasi-hyperplane of $U$ with respect to $\gamma_{p,q}$.
\end{definition}

\begin{theorem}\label{th3.1}
Let $C, U$ be defined in Lemma \ref{roth1}. Then for any given $y \in U\backslash \mbox{cl}(C)$, there exists a quasi-hyperplane separating	
$y$ and $C$, i.e.,  for any $z\in \mbox{int}(\mbox{cl}(C))$ (not necessarily equal to $C$), there exist $p\in \mbox{cl}(C)$ with $u=\exp^{-1}_p y$ and a real number $\alpha$ such that
$$\left\langle u, \exp^{-1}_p z\right\rangle_{T_p M}\leq\alpha<\left\langle u, \exp^{-1}_p y\right\rangle_{T_p M}.$$
\end{theorem}
\begin{proof}
Since $\mbox{cl}(C)$ is a closed locally geodesic convex set, for any $p, z\in \mbox{cl}(C),$ the unique minimal geodesic $\gamma_{p, z}$ from $p$ to $z$ is in $\mbox{cl}(C)$. Let $v=\dot{\gamma}_{p, z}(0)=\exp^{-1}_p z$. Then for some small enough $t>0$, we have
$$\exp_p t\frac{v}{\|v\|}=\exp_p \frac{t}{\|v\|}\exp^{-1}_p z=\gamma_{p, z}\left(\frac{t}{\|v\|}\right)\in \mbox{int}(\mbox{cl}(C)).$$
By Lemma \ref{roth1} and (\ref{eq01}), one has
\begin{equation}\label{e3.1}
\left\langle \exp^{-1}_p y, \exp^{-1}_p z \right\rangle_{T_p M}=\left\langle u, \exp^{-1}_p z \right\rangle_{T_p M}\leq 0,
\end{equation}
where $p=\mbox{Pr}_{\mbox{cl}(C)}(y)$ is the projection of $y$ on $\mbox{cl}(C)$.

On the other hand, $y\notin \mbox{cl}(C)$ implies that $p\neq y$ and $\dot{\gamma}_{p, y}(0)\neq 0$, and so
\begin{equation}\label{e3.2}
0<\|\dot{\gamma}_{p, y}(0)\|^2=\|u\|^2=\left\langle u, \exp^{-1}_p y\right\rangle_{T_p M}.
\end{equation}
It follows from (\ref{e3.1}) and (\ref{e3.2}) that
\begin{equation}\label{sep}
\left\langle u, \exp^{-1}_p z\right\rangle_{T_p M}\leq 0 <\left\langle u, \exp^{-1}_p y\right\rangle_{T_p M}.
\end{equation}
It means that $y$ and $C$ are separated by the quasi-hyperplane $H(\mbox{Pr}_{\mbox{cl}(C)}(y),y,0)$.
\end{proof}

Obviously, this separation property has the following relationship with the classic separation property of convex sets on the tangent space.

\begin{corollary}\label{c3.1}
Let $C, U, y, p$ be defined in Theorem \ref{th3.1}. If $H(p,y,0)$ is the separating quasi-hyperplane of $y$ and $C,$ then $\exp_p^{-1}(H(p,y,0))$ is the separating hyperplane of $\exp_p^{-1}y$ and $T_{\mbox{cl}(C)}(p).$
\end{corollary}

Theorem \ref{th3.1} shows that $p$ is in the separating quasi-hyperplane $H(\mbox{Pr}_{\mbox{cl}(C)}(y),y,0)$ and  $\left\langle \exp^{-1}_p y, \exp^{-1}_p z \right\rangle_{T_p M}\leq 0$ holds for all $z\in \mbox{int}(\mbox{cl}(C))$. Thus, we can call $H(\mbox{Pr}_{\mbox{cl}(C)}(y),y,0)$ the supporting quasi-hyperplane of $C$ at $p.$

The following theorem gives the existence of the supporting quasi-hyperplane of $C$ at any point of $\mbox{bd}(C).$
\begin{theorem}\label{th3.2}
	Let $C, U$ be defined in Lemma \ref{roth1}. Then for each $p\in \mbox{bd}(\mbox{cl}(C))$, there exists $y \in U\backslash \mbox{cl}(C)$ such that $\left\langle \exp^{-1}_p y, \exp^{-1}_p z \right\rangle_{T_p M}\leq 0$ for all $z\in \mbox{int}(\mbox{cl}(C))$, i.e.,
$$
H(p,y,0)=\left\{a\in U| \left\langle \exp^{-1}_p y, \exp^{-1}_p a\right\rangle_{T_p M}=0\right\}
$$
is a supporting quasi-hyperplane of $C$ at $p$.
\end{theorem}

\begin{proof}
Since $p\in \mbox{bd}(\mbox{cl}(C)),$ for any given $\delta_0>0,$ we have $\{y\in U|d(p,y)<\delta_0\}\cap U\backslash \mbox{cl}(C)\neq\emptyset$. Thus we can take a positive number sequence $\delta_k\rightarrow 0_+$ and a sequence $\{y^k\}$ with $y^k\in \{y\in U|d(p,y)<\delta_k\}\cap U\backslash \mbox{cl}(C)$ such that $y^k\rightarrow p$.  By Lemma \ref{rodf1}, there exists a unique geodesic $\gamma_{p, y^k}(t)$ from $p$ to $y^k$ with $\dot{\gamma}_{p,y^k}(0)=\exp_p^{-1}y^k$ and $\gamma_{p, y^k}(\delta_k)=y^k.$ By letting $p^k=\mbox{Pr}_{\mbox{cl}(C)}(y^k)$ and $u_k=\frac{\exp^{-1}_{p^k} y^k}{\|\exp^{-1}_{p^k} y^k\|},$ it follows from Theorem \ref{th3.1} that
\begin{equation}\label{e3.3}
	\left\langle u_k, \exp^{-1}_{p^k} z \right\rangle_{T_{p^k} M}=\frac{1}{\|\exp^{-1}_{p^k} y^k\|}\langle \exp^{-1}_{p^k} y^k, \exp^{-1}_{p^k} z \rangle_{T_{p^k} M}\leq 0, \quad \forall z\in \mbox{int}(\mbox{cl}(C)).
\end{equation}
Clearly, $\{u_k\}$ is a bounded sequence and so we can find $u\in T_p M$ with $\|u\|=1$ such that $u_k\rightarrow u$ as $k\rightarrow +\infty$ along a subsequence. By taking the limit to the subsequence in (\ref{e3.3}), and by using Lemma \ref{l2.1}, we obtain
\begin{equation}\label{e3.4}
\left\langle u, \exp^{-1}_{p} z \right\rangle_{T_{p} M}\leq 0, \quad \forall z\in \mbox{int}(\mbox{cl}(C)).
\end{equation}
It means that the normal cone is not empty at any point on the boundary of $\mbox{C}.$ Based on Lemma \ref{roth2} (ii), $0\neq u\in N_C(p)$ and there exists $t_1>0$ such that $\exp_p tu\in U$ and $\mbox{Pr}_C(\exp_p tu)=p$ on $t\in [0, t_1].$ Letting $t_0\in (0, t_1],$ then $y^*=\exp_p t_0 u\in U\backslash \mbox{cl}(C)$ and $p$ is the projection of $y^*$. It follows from (\ref{e3.4}) that
$$\left\langle \exp_p^{-1}y^*, \exp^{-1}_{p} z \right\rangle_{T_{p} M}=\left\langle \exp_p^{-1}\exp_p t_0 u, \exp^{-1}_{p} z \right\rangle_{T_{p} M}=t_0\left\langle u, \exp^{-1}_{p} z \right\rangle_{T_{p} M}\leq 0,\quad \forall z\in \mbox{int}(\mbox{cl}(C)).$$
This completes the proof.
\end{proof}

\section{Applications}
\subsection{Tangent Cone and Normal Cone}\noindent
\setcounter{equation}{0}
In order to apply the separation theorem to obtain some optimal conditions for optimization problems with constraints on Riemannian manifolds, we discuss some properties of tangent cone and normal cone of a locally geodesic convex set.
\begin{lemma}\label{l4.1}
Let $C, D$ be locally geodesic convex subsets of $M$.  If $C\subset D,$ then $N_D(p)\subset N_C(p)$ for all $p\in C$.
\end{lemma}
\begin{proof} Clearly, $C\subset D$ implies that $T_C(p)\subset T_D(p)$. Let $u\in N_D(p)$. Then $\left\langle u, v \right\rangle_{T_p M}\leq 0$ for all $v\in T_D(p).$ Thus, for any $v'\in T_C(p)\subset T_D(p),$ we have $\left\langle u, v' \right\rangle_{T_p M}\leq 0$  and so $u\in N_C(p)$. This shows that $N_D(p)\subset N_C(p)$.
\end{proof}

\begin{lemma}\label{l4.2}
	Let $C$ be a locally geodesic convex subset of $M$. Then for any $\bar{x}\in C$,  $y\in T_C(\bar{x})$ if and only if there exist a sequence $\{x^k\}$ in $\mbox{int}(C)$ with $x^k\to \bar{x}$ and a nonnegative number sequence $\{\alpha_k\}$ such that $$y=\lim_{k\rightarrow +\infty}\alpha_k\exp^{-1}_{\bar{x}}x^k.$$
\end{lemma}
\begin{proof} Let $y\in T_C(\bar{x})$. Then there is $t<\varepsilon(\bar{x})$ such that $\exp_{\bar{x}}t\frac{y}{\|y\|}\in \mbox{int}(C).$ Let $x^k=\exp_{\bar{x}}t_k\frac{y}{\|y\|}$ with $t_k\leq t$. Then $x^k\rightarrow \bar{x}$ as $t_k\rightarrow 0_+$. Putting $\alpha_k=\frac{\|y\|}{t_k}\geq 0,$ we have
$$\lim_{k\rightarrow +\infty}\alpha_k\exp^{-1}_{\bar{x}}x^k=\lim_{k\rightarrow +\infty}\alpha_k\exp^{-1}_{\bar{x}}\exp_{\bar{x}}t_k\frac{y}{\|y\|}=y.$$

Conversely, for any given $x^0\in \mbox{int}(C),$ there exists a unique geodesic from $\bar{x}$ to $x^0,$ namely $\gamma_{\bar{x}, x^0}(t),$ in the locally geodesic convex set $C.$ Let $\gamma_{\bar{x}, x^0}(0)=\bar{x}$ and $x^k=\gamma_{\bar{x}, x^0}(\frac{1}{k+1})$ for $k=0,1,2,\cdots$. Then $x^k\rightarrow \bar{x}$ as $k\rightarrow +\infty$ and $x^k\in \mbox{int}(C).$ If $\alpha_k=\|\exp_{\bar{x}}^{-1}x^k\|\geq 0,$ then
 $$y=\lim_{k\rightarrow +\infty}\alpha_k\exp^{-1}_{\bar{x}}x^k=\lim_{k\rightarrow +\infty}\frac{\exp_{\bar{x}}^{-1}x^k}{\|\exp_{\bar{x}}^{-1}x^k\|}=\lim_{k\rightarrow +\infty}\frac{\dot{\gamma}_{\bar{x},x^0}(\frac{1}{k+1})}{\|\dot{\gamma}_{\bar{x},x^0}(\frac{1}{k+1})\|}
 =\frac{\dot{\gamma}_{\bar{x},x^0}(0)}{\|\dot{\gamma}_{\bar{x},x^0}(0)\|}
 =\frac{\exp_{\bar{x}}^{-1}x^0}{\|\exp_{\bar{x}}^{-1}x^0\|}$$
 and so
 $$
 \exp_{\bar{x}}t \frac{y}{\|y\|}=\exp_{\bar{x}} t \frac{\exp_{\bar{x}}^{-1}x^0}{\|\exp_{\bar{x}}^{-1}x^0\|}=x^0\in \mbox{int}(C),
 $$
where $t=\frac{1}{\|\exp_{\bar{x}}^{-1}x^0\|}>0$. This implies that $y\in T_{C}(\bar{x})$.
\end{proof}

Now we can show the following properties of tangent cone and normal cone of a locally geodesic convex set on Riemannian manifolds.
\begin{theorem}\label{th4.1}
	For any given $p\in C\subset M$, the tangent cone $T_C(p)$ is a nonempty closed subset of $T_pM$, and $T_C(p)=\mbox{cl}(T_C(p))=(N_C(p))^*.$
\end{theorem}
\begin{proof} First we prove that $T_C(p)$ is a nonempty closed set. In fact, it follows from $0\in T_C(p)$ that $T_C(p)$ is nonempty. Let $\{y^i\}\subset T_C(p)$ be a sequence with $y^i\to y^*$.  We need to prove that $y^*\in T_C(p)$. To this end, by Lemma \ref{l4.2}, for any given $i$,  there exist a sequence $\{x^k(i)\}\subset \mbox{int}(C)$ with $x^k(i) \to p$ and a nonnegative number sequence $\{\alpha_k(i)\}$ such that $y^i=\lim_{k\rightarrow +\infty}\alpha_k(i)\exp^{-1}_{p}x^k(i)$. Thus, there exists $k_1(i)$ such that for any $k\geq k_1(i),$ $d(x^k(i),p)\leq \frac{1}{i},$ and exists $k_2(i)$ such that, for any $k\geq k_2(i),$ 
$$\|\alpha_k(i)\exp^{-1}_{p}x^k(i)-y^i\|\leq \frac{1}{i}.$$
Let $k(i)=\max\{k_1(i), k_2(i)\}$,  $\alpha_i=\alpha_{k_i}(i)$, and  $x_i=x_{k_i}(i)$.  Then $d(x_i, p)\leq \frac{1}{i}\rightarrow 0$ and $$\|\alpha_i\exp_p^{-1}x^i-y^*\|\leq\|\alpha_i\exp_p^{-1}x^i-y^i\|+\|y^i-y^*\|\leq\frac{1}{i}+\|y^i-y^*\|\rightarrow 0$$ as $i\rightarrow +\infty.$ Therefore, $y^*\in T_C(p).$

Then we show that $\mbox{cl}(T_C(p))\subset(N_C(p))^*$. For each $u\in T_C(p)$, one has $\left\langle u, v\right\rangle_{T_{p} M}\leq 0$ for all $v\in N_C(p)$  and so $u\in (N_C(p))^*$.  This shows that $T_C(p)\subset (N_C(p))^*$. Since $(N_C(p))^*$ is closed, we have $\mbox{cl}(T_C(p))\subset(N_C(p))^*$.

Next we prove that $\mbox{cl}(T_C(p))\supset (N_C(p))^*$. Assume that there exists $\hat{x}$ belongs to $(N_C(p))^*$ but not to $\mbox{cl}(T_C(p))$. Then it follows from the definition of tangent cone that $\exp_p\hat{x}\notin \mbox{cl}(C).$ From (\ref{sep}) in Theorem \ref{th3.1}, one has
$$\left\langle \hat{x},\exp_p^{-1}z\right\rangle_{T_{p} M}\leq 0, \quad \forall z\in \mbox{int}(\mbox{cl}(C)),$$
which shows that $\hat{x}\in N_{\mbox{cl}(C)}(p)\subset N_{C}(p)$ by Lemma \ref{l4.1}. Since $0\in \mbox{cl}(T_C(p)),$ so $\hat{x}\neq 0$ and it contradicts the fact $\hat{x}\in N_{C}(p)\cap (N_C(p))^*$.

By the facts $\mbox{cl}(T_C(p))\subset(N_C(p))^*$ and $\mbox{cl}(T_C(p))\supset (N_C(p))^*$, we have $\mbox{cl}(T_C(p))=(N_C(p))^*$ and so $T_C(p)=\mbox{cl}(T_C(p))=(N_C(p))^*$.
\end{proof}

\subsection{Optimality conditions of optimization problems with constraints}
In this subsection, we consider the following manifold optimization problem (for short, MOP) on a Riemannian manifold $M$:
\begin{eqnarray}\label{mop}
\left\{
\begin{array}{l}
	\min f(x)\\
	\mbox{s.t.}\quad x\in S,
\end{array}
\right.
\end{eqnarray}
where $S\subset M$ is a nonempty geodesic convex subset and $f: M\rightarrow \mathbb{R}$ is a differentiable function.

Based on Theorem \ref{th4.1}, we can obtain the necessary conditions for the local optimal solution of the MOP \eqref{mop} as follows.
\begin{theorem}\label{t4.2}
	If $\bar{x}$ is the local optimal solution of the MOP \eqref{mop} and $S$ is a locally geodesic convex subset of $M$, then $-\mbox{grad}f(\bar{x})\in N_S(\bar{x}).$
\end{theorem}
\begin{proof} From Lemma \ref{l4.2}, for any $y\in T_S(\bar{x}),$ there exist a sequence $\{x^k\}\subset \mbox{int}(S)$ with  $x^k\to \bar{x}$ and a nonnegative number sequence $\{\alpha_k\}$ such that $y=\lim_{k\rightarrow +\infty}\alpha_k\exp^{-1}_{\bar{x}}x^k.$ Since $f$ is differentiable at $\bar{x},$ we have
$$
\langle \mbox{grad}f(\bar{x}), \dot{\gamma}_{\bar{x},x^k}(0)\rangle_{T_{\bar{x}}M}=\lim_{t\rightarrow 0+}\frac{f(\gamma_{\bar{x},x^k}(t))-f(\bar{x})}{t},
$$
where $x^k\in \mbox{int}(S)$ and $\gamma_{\bar{x},x^k}(t)$ is the unique geodesic from $\bar{x}$ to $x^k.$ Because $S$ is a locally geodesic convex set and $x^k\in \mbox{int}(S),$  we know that $\gamma_{\bar{x},x^k}(t)$ lies in $\mbox{int}(S)$ as $0<t\leq 1.$

If $\bar{x}$ is the local optimal solution, then $f(x)\geq f(\bar{x})$ for any $x\in S$ and so
$$\langle -\mbox{grad}f(v), \dot{\gamma}_{\bar{x},x^k}(0)\rangle_{T_{\bar{x}}M}=\langle -\mbox{grad}f(\bar{x}), \exp_{\bar{x}}^{-1}x^k\rangle_{T_{\bar{x}}M}\leq 0, \forall x^k\in \mbox{int}(S).$$
Since $\alpha_k\geq 0$, it follows from Theorem \ref{th4.1} that $T_S(\bar{x})$ is closed and so
\begin{equation}\label{e4.1}
	\langle -\mbox{grad}f(\bar{x}), \lim_{k\rightarrow +\infty}\alpha_k\exp_{\bar{x}}^{-1}x^k\rangle_{T_{\bar{x}}M}=\langle -\mbox{grad}f(\bar{x}), y\rangle_{T_{\bar{x}}M}\leq 0, \quad \forall y\in T_S(\bar{x}).
\end{equation}
This implies that $-\mbox{grad}f(\bar{x})\in N_S(\bar{x})$.
\end{proof}

Next, we consider how the MOP \eqref{mop} can be further characterized if the feasible set $S$ is defined by a set of inequality constraints. More precisely, we consider the following MOP with a set of inequality constraints (for short, IEQMOP):
\begin{eqnarray}\label{ieqmop}
\left\{
\begin{array}{l}
	\min f(x)\\
	\mbox{s.t.} \quad  x\in S=\{x|g_i(x)\leq 0, i=1,2,\ldots, l\},
\end{array}
\right.
\end{eqnarray}
where $g_i: M\rightarrow \mathbb{R},i=1,2,\ldots, l$ are geodesic convex functions.
Let $S_j=\{x\in M|g_j(x)\leq 0\}$ and $S=\cap_{i=1}^l S_i$. Clearly, for each $i=1,2,\ldots, l$, $S_i$ is a locally geodesic convex set by the definition of the geodesic convex function and so the feasible set $S$ is also a locally geodesic convex set. The constraints that satisfy $g(\bar{x})=0$ at a feasible solution $\bar{x}$ are called active constraints, and their index set is denoted as $I(\bar{x})=\{i|g_i(\bar{x})=0\}$. If the function $g$ is differentiable at $\bar{x}$, then the following linearization cone $C_S(\bar{x})$ can be defined:
$$
C_S(\bar{x})=\{y\in T_{\bar{x}}M|\langle\mbox{grad}g_i(\bar{x}),y\rangle_{T_{\bar{x}}M}\leq 0, i\in I(\bar{x}) \}.
$$
Thus, both the tangent cone and linearization cone are approximate descriptions on the tangent space of the geodesic convex set at any point in it, and they have the following relationship.
\begin{lemma}\label{l4.3}
	For any point $\bar{x}\in S$ in IEQMOP \eqref{ieqmop}, one has $T_S(\bar{x})\subset C_S(\bar{x}).$
\end{lemma}
\begin{proof} Let $y\in T_S(\bar{x})$. Then there exist a sequence $\{x^k\}\subset \mbox{int}(S)$ with $x^k\to \bar{x}$ and a nonnegative number sequence $\{\alpha_k\}$ such that $y=\lim_{k\rightarrow +\infty}\alpha_k\exp^{-1}_{\bar{x}}x^k.$ Since $x^k\in \mbox{int}(S)$ and $g_i(\bar{x})=0$ for $i\in I(\bar{x})$, one has
\begin{equation}\label{e4.2}
	\langle\mbox{grad}g_i(\bar{x}),\exp_{\bar{x}}^{-1}x^k\rangle_{T_{\bar{x}}M}=\lim_{t\rightarrow 0+}\frac{g_i(\exp_{\bar{x}}t\exp_{\bar{x}}^{-1}x^k)-g_i(\bar{x})}{t}=\lim_{t\rightarrow 0+}\frac{g_i(\exp_{\bar{x}}t\exp_{\bar{x}}^{-1}x^k)-0}{t}\leq 0.
\end{equation}
Taking $k\rightarrow +\infty$ in (\ref{e4.2}), since $\alpha_k\geq 0$, we have
$$\langle \mbox{grad}g_i(\bar{x}), \lim_{k\rightarrow +\infty}\alpha_k\exp_{\bar{x}}^{-1}x^k\rangle_{T_{\bar{x}}M}=\langle \mbox{grad}g_i(\bar{x}), y\rangle_{T_{\bar{x}}M}\leq 0, \quad i\in I(\bar{x})$$
and so $y\in C_S(\bar{x})$.
\end{proof}

\begin{lemma}\label{l4.4}
The polar cone
$$
C_S(\bar{x})^*=\left\{\left.\sum_{i\in I(\bar{x})}\lambda_i \mbox{grad}g_i(\bar{x})\right|\lambda_i\geq 0\right\}.
$$
\end{lemma}

\begin{proof}
Let $K=\left\{\sum_{i\in I(\bar{x})}\lambda_i \mbox{grad}g_i(\bar{x})|\lambda_i\geq 0\right\}$. Then it easy to see that $K$ is a closed convex set in $T_{\bar{x}}M.$ For any $y\in C_S(\bar{x}),$ we obtain
 $$\big\langle \sum_{i\in I(\bar{x})}\lambda_i \mbox{grad}g_i(\bar{x}), y\big\rangle_{T_{\bar{x}}M}=\sum_{i\in I(\bar{x})}\lambda_i\langle \mbox{grad}g_i(\bar{x}), y\rangle_{T_{\bar{x}}M}\leq 0.$$
Thus, $\sum_{i\in I(\bar{x})}\lambda_i \mbox{grad}g_i(\bar{x})\in C_S(\bar{x})^*$ and so $y\in K^*,$ which means that $K\subset C_S(\bar{x})^*$ and $C_S(\bar{x})\subset K^*.$

On the other hand, for any $x\in K^*$ and any $\lambda_i\geq 0,$ one has
$$\left\langle x, \sum_{i\in I(\bar{x})}\lambda_i \mbox{grad}g_i(\bar{x})\right\rangle_{T_{\bar{x}}M}=\sum_{i\in I(\bar{x})}\lambda_i\langle x, \mbox{grad}g_i(\bar{x})\rangle_{T_{\bar{x}}M}\leq 0.$$
From the arbitrariness of $\lambda_i,$ we known that $\langle x, \mbox{grad}g_i(\bar{x})\rangle_{T_{\bar{x}}M}\leq 0$. Thus $x\in C_S(\bar{x})$ and so $K^*\subset C_S(\bar{x}).$ By the closed convexity of $K,$ we have $C_S(\bar{x})^*=K^{**}=K.$
\end{proof}

\begin{theorem}
Let $\bar{x}$ be the local optimal solution of IEQMOP \eqref{ieqmop}. Then there exist $\bar{\lambda}_0,\bar{\lambda}_1,\ldots,\bar{\lambda}_l$ such that
$$
\left\{
\begin{array}{l}
\bar{\lambda}_0\mbox{grad}f(\bar{x})+\sum_{i=1}^l \bar{\lambda}_i\mbox{grad}g_i(\bar{x})=0;\\
g_i(\bar{x})\leq 0, \; \bar{\lambda}_i g_i(\bar{x})=0, \; i=1,\ldots, l;\\
\bar{\lambda}_i\geq 0, \; i=1,\ldots, l.
\end{array}
\right.
$$
\end{theorem}

\begin{proof}
By (\ref{e4.1}) and Lemma \ref{l4.3}, for any $y\in T_S(\bar{x}),$ we have
$$
\langle -\mbox{grad}f(\bar{x}),y\rangle_{T_{\bar{x}}M}\leq 0,\quad
\langle \mbox{grad}g_i(\bar{x}),y\rangle_{T_{\bar{x}}M}\leq 0,\quad  i\in I(\bar{x})=\{i|g_i(\bar{x})=0\}
$$
and so
$$\{y\in T_{\bar{x}}M|\langle \mbox{grad}f(\bar{x}),y\rangle_{T_{\bar{x}}M}\leq0\}\cap \{y\in T_{\bar{x}}M|\langle \mbox{grad}g_i(\bar{x}),y\rangle_{T_{\bar{x}}M}\leq0, i\in I(\bar{x})\}=\{0\}.$$
Define a set $C$ by setting
$$C=\{(y,a)\in T_{\bar{x}}M\times \mathbb{R}|a+\langle \mbox{grad}f(\bar{x}),y\rangle_{T_{\bar{x}}M}\leq 0, a+\langle \mbox{grad}g_i(\bar{x}),y\rangle_{T_{\bar{x}}M}\leq 0, i\in I(\bar{x})\}.$$
Then it is easy to see that $(0,1)\cdot(y,a)^T=a\leq 0$ for all $(y, a)\in C$ and so $(0,1)\in C^*.$ From Lemma \ref{l4.4}, we can immediately obtain
$$
\left\{
\begin{array}{l}
	\bar{\lambda}_0\mbox{grad}f(\bar{x})+\sum_{i\in I(\bar{x}))} \bar{\lambda}_i\mbox{grad}g_i(\bar{x})=0;\\
	\bar{\lambda}_0+ \sum_{i\in I(\bar{x}))} \bar{\lambda}_i=1,\bar{\lambda}_i\geq 0, i\in \{0\}\cup I(\bar{x}).
\end{array}
\right.$$
This completes the proof.
\end{proof}

If some suitable constraint qualifications are added, we can obtain the following Karush-Kuhn-Tucker conditions of IEQMOP \eqref{ieqmop}.
\begin{theorem}
	Let $\bar{x}$ be the local optimal solution of (IEQMOP). If $C_S(\bar{x})\subset \mbox{co}\;T_S(\bar{x})$, then there exist $\bar{\lambda}_1,\bar{\lambda}_2,\ldots,\bar{\lambda}_l$ such that
	$$
	\left\{
	\begin{array}{l}
		\mbox{grad}f(\bar{x})+\sum_{i=1}^l \bar{\lambda}_i\mbox{grad}g_i(\bar{x})=0;\\
		\bar{\lambda}_i\geq 0, \; g_i(\bar{x})\leq 0, \; \bar{\lambda}_i g_i(\bar{x})=0,\; i=1,\ldots, l.
	\end{array}
	\right.
	$$
\end{theorem}

\begin{proof} Since polar cone is a convex closed subset in $T_{\bar{x}}M,$ we have
$$
N_S(\bar{x})=T_S(\bar{x})^*=(\mbox{co}\;T_S(\bar{x}))^*.
$$
From Lemma \ref{l4.1} and Theorem \ref{t4.2}, if $\bar{x}$ is the local optimal solution, then
$$
-\mbox{grad}f(\bar{x})\in N_S(\bar{x})=(\mbox{co}\;T_S(\bar{x}))^*\subset C_S(\bar{x})^*.
$$
By Lemma \ref{l4.4}, one has
$$-\mbox{grad}f(\bar{x})\in C_S(\bar{x})^*=\left\{\left.\sum_{i\in I(\bar{x})}\bar{\lambda}_i \mbox{grad}g_i(\bar{x})\right|\bar{\lambda}_i\geq 0\right\}.$$
It means that there exists $\bar{\lambda}_i\geq 0$ with $i\in I(\bar{x})$ such that
$$\mbox{grad}f(\bar{x})+\sum_{i\in I(\bar{x})}\bar{\lambda}_i \mbox{grad}g_i(\bar{x})=0.$$
Let $\bar{\lambda}_i\geq 0$ if $i\in I(\bar{x})$ and $\bar{\lambda}_i= 0$ if $i\notin I(\bar{x})$. Then the Karush-Kuhn-Tucker conditions of IEQMOP \eqref{ieqmop} is given.
\end{proof}

\section{Conclusions}\noindent
\setcounter{equation}{0}
As an attempt of exploring the geodesic convexity, we obtained the separation theorem of geodesic convex sets using the property of projection. As applications, the optimality conditions of constrained optimization problems on manifolds is derived from the theory of cones, which is different from the previous studies. The method presented in this paper can help us to conduct more generalized research, such as the case where the objective function is a manifold mapping to another manifold. We foresee further progress in this topic in the nearby future. Here we have one open questions:
How to give the algebraic form of the separation theorem of two disjoint geodesic convex sets on a manifold?
	
%\noindent{\bf Data availability statement}

%No data were used to support this study.

%\noindent{\bf Conflicts of interest}

%We declare that there are no conflicts of interest regarding the publication of this paper.

\begin{acknowledgement}
This work was supported by the National Natural Science Foundation of China (No: 11901484, 12171377), Science and Technology Innovation Seedling Project Funding Program of Sichuan Province (No: 2022034), Natural Science Foundation of Sichuan Province (Youth) (No. 2022NSFSC1834) and Natural Science Starting Project of SWPU (No. 2023QHZ007).
\end{acknowledgement}

\end{document}